\documentclass[12pt, reqno]{amsart}
\usepackage{amsmath}
\usepackage{amssymb}
\usepackage{epsfig}
\usepackage{graphicx}
\usepackage{color}
\usepackage{fullpage}
\usepackage{comment}
\definecolor{shadecolor}{gray}{0.875}
\usepackage{amscd}
\usepackage{stmaryrd}
\usepackage{threeparttable}

\usepackage{ulem}
\usepackage{amsthm}
\usepackage{csquotes}
\usepackage{enumitem}
\usepackage{graphics}
\usepackage{tikz-cd}
\usepackage{tabularx}
\usepackage{hhline}
\usepackage{multirow}
\usepackage{makecell}
\usepackage{booktabs}
\usetikzlibrary{backgrounds,calc,positioning}
\usepackage{multicol} 

\usepackage[backend=bibtex,style=alphabetic,maxbibnames=99,url=false]{biblatex}
\usepackage[colorlinks=false,urlbordercolor=white]{hyperref}

\tikzset{sgplattice/.style={inner sep=1pt,norm/.style={red!50!blue},char/.style={blue!50!black},
		lin/.style={black!50}},cnj/.style={black!50,yshift=-2.5pt,left=-1pt of #1,scale=0.5,fill=white}}
	
	\usepackage{here}
	\usepackage{caption}
	\usepackage{subcaption}

\let\oldtocsection=\tocsection

\let\oldtocsubsection=\tocsubsection

\let\oldtocsubsubsection=\tocsubsubsection

\renewcommand{\tocsection}[2]{\hspace{0em}\oldtocsection{#1}{#2}}
\renewcommand{\tocsubsection}[2]{\hspace{1em}\oldtocsubsection{#1}{#2}}
\renewcommand{\tocsubsubsection}[2]{\hspace{2em}\oldtocsubsubsection{#1}{#2}}

\makeatletter

\newcommand{\Rmnum}[1]{\expandafter\@slowromancap\romannumeral #1@}
\makeatother

\numberwithin{equation}{section}

\input xy
\xyoption{all}

\calclayout
\allowdisplaybreaks[3]
\theoremstyle{plain}
\newtheorem{prop}{Proposition}[section]

\newtheorem{lemm}[prop]{Lemma}

\theoremstyle{definition}
\newtheorem{defi}[prop]{Definition}

\newtheorem{prob}[prop]{Problem}

\newtheorem{rema}[prop]{Remark}

\newtheorem{mainthm}{Theorem}

\newlist{steps}{enumerate}{1}
\setlist[steps, 1]{label = Step \arabic*:}

\def\ra{\rightarrow}

\def\cG{{\mathcal G}}
\def\cH{{\mathcal H}}

\def\cK{{\mathcal K}}
\def\cL{{\mathcal L}}
\def\cM{{\mathcal M}}

\def\cO{{\mathcal O}}
\def\cP{{\mathcal P}}

\def\cS{{\mathcal S}}

\def\cV{{\mathcal V}}
\def\cW{{\mathcal W}}

\def\bC{{\mathbb C}}

\def\bF{{\mathbb F}}

\def\bP{{\mathbb P}}
\def\bQ{{\mathbb Q}}
\def\bR{{\mathbb R}}

\def\H{\mathrm{H}}

\def\diag{\mathrm{diag}}

\def\Spec{\mathrm{Spec}}

\def\GL{\mathrm{GL}}
\def\PGL{\mathrm{PGL}}

\def\Gal{\mathrm{Gal}}
\def\Aut{\mathrm{Aut}}

\makeatother
\makeatletter

\author{Tianzhi Yang}
\address{Piazza dei Cavalieri, 7, 56126 Pisa, Italy}
\email{tianzhi.yang@sns.it}

\title{Fields of Moduli of Smooth Del Pezzo Surfaces}

\begin{document}
\date{\today}
\begin{abstract}
The field of moduli of a variety $X$ over an algebraically closed field $K$ is defined as the fixed field of those automorphisms $\sigma$ of $K$ for which $X\simeq X^{\sigma}$. A fundamental question is under what conditions a variety admits a model over its field of moduli. We give a complete answer for smooth del Pezzo surfaces in characteristic $0$: every smooth del Pezzo surface of degree at least $3$ has a model over its field of moduli, whereas in degrees $1$ and $2$ there exist smooth complex del Pezzo surfaces with field of moduli $\bR$ which do not admit a real model.
\end{abstract}

\maketitle
\setcounter{tocdepth}{2}

\section{Introduction}
Let $k$ be a field of characteristic $0$ and $K$ its algebraic closure. Let $X$ be a variety over $K$. Consider the subgroup $\Delta \subset \operatorname{Gal}(K/k)$ consisting of those $\sigma \in \operatorname{Gal}(K/k)$ for which $X^{\sigma} \simeq X$, where $X^{\sigma}$ denotes the Galois conjugate of $X$. Concretely, if $X$ is given by a system of polynomial equations with coefficients in $K$, then $X^{\sigma}$ is obtained by applying $\sigma$ to all those coefficients. One can verify that $\Delta$ is an open subgroup, and therefore corresponds to a fixed field $k_X\subset K$, with $k_X/k$ finite. This field $k_X$ is called the \emph{field of moduli} of $X$.

The notion of field of moduli was first introduced by Matsusaka in \cite{Matsusaka1958}. A key property is that for any intermediate field $K/h/k$, if $X$ admits a model over $h$—that is, if there exists a variety $\mathfrak{X}$ over $h$ with $\mathfrak{X}_K \simeq X$—then we necessarily have $k_X \subset h$. Thus, the field of moduli is contained in every field of definition of $X$. This naturally leads to the following question:

\begin{prob}\label{main-problem}
    Is the field of moduli of $X$ itself a field of definition?
\end{prob}

Foundational contributions to this subject were made by A. Weil \cite{Weil1956}, T. Matsusaka \cite{Matsusaka1958}, and G. Shimura \cite{Shimura1959}, and the problem has since remained an active focus of research; for instance, in the context of curves and abelian varieties, one can consult \cite{Murabayashi1996}, \cite{DebesDouai1997}, \cite{DebesEmsalem1999}, \cite{CardonaQuer2005}, \cite{Huggins2007}, \cite{Kontogeorgis2009}, \cite{Hidalgo2009}, \cite{Marinatto2013}.

For smooth projective algebraic surfaces other than abelian surfaces, the field of moduli problem \ref{main-problem} remains largely undeveloped. $K3$ surfaces form the principal explicitly studied class: Laface computes relative and absolute fields of moduli for singular $K3$ surfaces through transcendental lattices and ring class fields
\cite{Laface2016}, while Valloni develops analogous class field theoretic results and fields of definition for $CM$ $K3$ surfaces \cite{Valloni2021,Valloni2023}; systematic results for other surfaces remain mostly open.

In this paper, we study Problem~\ref{main-problem} for smooth del Pezzo surfaces over $K$. Recall that the degree of a del Pezzo surface $S$ is the integer $d=K_S^2$, with $1\leq d\leq 9$.
\begin{mainthm}
Let $k$ be a field of characteristic $0$ with algebraic closure $K$.
\begin{enumerate}[label={\rm(\arabic*)}]
\item Every smooth del Pezzo surface over $K$ of degree at least $3$ is defined over its field of moduli relative to $K/k$.
\item In each of degrees $1$ and $2$, there exists a smooth del Pezzo surface over $\bC$ whose field of moduli relative to $\bC/\bR$ is $\bR$, but which is not defined over $\bR$.
\end{enumerate}
\end{mainthm}

In particular, if $3\leq d\leq 9$, let $\cM_d$ be the moduli stack of smooth del Pezzo surfaces of degree $d$ over $k$, and let $\cM_d\ra M_d$ be the coarse moduli space. Then the map
\[
\cM_d(h)\ra M_d(h)
\]
is surjective for every field extension $h/k$.

The proof is divided according to the degree. In degrees at least $5$, there are no moduli and the result follows from the standard models over $\bQ$. For degrees $1\leq d\leq4$, we use the gerbe-theoretic formulation of the field-of-moduli obstruction developed by Bresciani and Vistoli~\cite{bresciani-vistoli}. A smooth del Pezzo surface $S$ of degree $d\leq4$ determines a finite residue gerbe $\cG_S$ over its field of moduli, and $S$ descends precisely when $\cG_S$ is neutral.

A del Pezzo surface of degree $4$ is an intersection of two quadrics in $\bP^4$. The five singular members of the pencil of quadrics form a reduced divisor of degree $5$ on $\bP^1$, and determine the surface up to isomorphism. Since a finite subset of odd cardinality of $\bP^1$ is defined over its field of moduli by~\cite{Marinatto2013}, the divisor descends; a trace construction then reconstructs the del Pezzo surface over the same field. For cubic surfaces, we also use the \emph{neutral representation} criteria developed in joint work with Bresciani~\cite{neutral-representation-dim-leq3}, which detect neutralness from low dimensional representations of the geometric inertia group. The cubic argument is based on the classification of automorphisms in~\cite{dolgachev-duncan}, together with the combinatorial structure of the $27$ exceptional lines, with particular emphasis on the Eckardt loci.

In degree two, the anticanonical morphism realizes the surface as a double cover of $\bP^2$ branched over a smooth plane quartic curve. Descending the surface is thus equivalent to descending its branch quartic, and the counterexamples constructed by Artebani and Quispe~\cite{artebani-quispe} show that this is not always possible. Moreover, by \cite[Corollary 5.7]{bresciani-plane}, a degree two del Pezzo surface does descend to its field of moduli provided that its associated quartic curve $C$ has automorphism group $\Aut(C)$ not isomorphic to $C_2$.

In degree one, we construct a new obstruction using the Bertini involution. For a general weighted equation invariant under a suitable semilinear transformation, the composite $\varphi\overline{\varphi}$ of every possible semilinear isomorphism $\varphi$ with its conjugate is the Bertini involution, and hence Weil's cocycle condition cannot be satisfied.

\subsection{Acknowledgement} I am grateful to Giulio Bresciani for many insightful discussions.

\section{Del Pezzo surfaces of degree at least five}
We first dispose of the degrees larger than $4$.

\begin{prop}\label{prop:degree-at-least-five}
Let $S$ be a smooth del Pezzo surface over $K$ of degree at least $5$. Then $S$ is defined over its field of moduli.
\end{prop}

\begin{proof}
Over an algebraically closed field, a del Pezzo surface of degree $9$ is isomorphic to $\bP^2$; in degree $8$ it is isomorphic to either $\bP^1\times\bP^1$ or $\bF_1$; and in degrees $7$, $6$, and $5$ there is a unique isomorphism class, obtained by blowing up respectively $2$, $3$, and $4$ points in general position in $\bP^2$. Each of these isomorphism classes has a model over $\bQ$. The field of moduli relative to $K/k$ is therefore $k$, and a standard model over $\bQ\subset k$ gives the result.
\end{proof}

\section{Residue gerbes and semilinear automorphisms}
Throughout this section, let $k$ be a field of characteristic $0$ with algebraic closure $K$, and fix an integer $1\leq d\leq4$. Let $\cM_d$ be the moduli stack of smooth del Pezzo surfaces of degree $d$ over $k$. For a smooth del Pezzo surface $S$ of degree $d$ over $K$, let
\[
\cG_S\subset \cM_d
\]
be the residue gerbe of the geometric point $[S]:\Spec K\to \cM_d$. More specifically, over a $k$-scheme $T$, an object $\mathfrak{S}\ra T$ in $\cG_S(T)$ is a smooth proper family of del Pezzo surfaces of degree $d$ over $T$, such that there exists an fppf cover $T'\ra T_K$ with a $T'$-isomorphism $\mathfrak{S}_{T'}\simeq S_{T'}$. Since $d\leq4$, the automorphism group of every geometric object of $\cM_d$ is finite. Thus $\cM_d$ is a Deligne--Mumford stack, and $\cG_S$ is a finite gerbe over the field of moduli $k_S$ of $S$, with geometric inertia group $\Aut(S)$; see \cite[Proposition~3.10]{bresciani-vistoli}. By the formalism, we have
\[
S\text{ descends to }k_S\quad\Longleftrightarrow\quad \cG_S(k_S)\neq \emptyset.
\]

From now on, we base change to $k=k_S$ for simplicity. Let $\Gamma$ be an algebraic group scheme over $k$ and let $G\subset\Gamma(K)$ be a finite subgroup, whose Galois conjugates are $\Gamma(K)$-conjugates. In order to study the neutralness of such a residue gerbe $\cG_S$, we introduce the following definition. 

\begin{defi}[{\cite[Definition~3.4]{neutral-representation-dim-leq3}}]\label{def:neutral}
The embedding $G\hookrightarrow\Gamma(K)$ is \emph{neutral} if every faithful morphism from a finite gerbe to $B\Gamma$ over $k$ whose geometric inertia embedding is conjugate to $G\subset \Gamma(K)$ has a $k$-rational point.
\end{defi}

The main idea of the paper is the following. If some additional structure on $S$ gives a faithful morphism
\[
\cG_S\to B\Gamma
\]
and the resulting geometric inertia embedding $\Aut(S)\hookrightarrow\Gamma(K)$ is neutral, then $\cG_S$ has a $k$-rational point, and hence $S$ descends to its field of moduli.

Put $G_k=\Gal(K/k)$. Since $k=k_S$, for every $\sigma\in G_k$ there is an isomorphism $S^{\sigma}\simeq S$. Define the group of semilinear automorphisms
\[
\widetilde G_S=
\bigl\{(\sigma,\varphi):\sigma\in G_k,
\ \varphi:S^{\sigma}\xrightarrow{\sim}S\bigr\},
\]
where the multiplication is induced by composition of the corresponding semilinear maps. There is an exact sequence
\begin{equation}\label{eq:semilinear-extension}
1\to\Aut_K(S)\to\widetilde G_S
\to G_k\to 1.
\end{equation}
Thus an element of $\widetilde G_S$ lying over $\sigma$ is precisely a $K$-isomorphism $S^{\sigma}\to S$.

In the usual quotient notation, the residue gerbe and its universal surface may be written as
\begin{equation}\label{eq:semilinear-quotients}
\cG_S=[\Spec K/\widetilde G_S],
\qquad
\cS=[S/\widetilde G_S].
\end{equation}
Here the brackets denote the stackification of the Galois descent groupoid; they should not be read as a quotient by a finite algebraic group.

The action of $\widetilde G_S$ on points of $S$ is given by
\[
(\sigma,\varphi)\cdot p = \varphi\bigl(p^{\sigma}\bigr).
\]
Accordingly, a point $p\in S(K)$ is \emph{distinguished} if
\begin{equation}\label{eq:intrinsic-point-galois}
\varphi(p^{\sigma})=p
\end{equation}
for every $\sigma\in G_k$ and every isomorphism $\varphi:S^{\sigma}\xrightarrow{\sim}S$; equivalently, $p$ is fixed by all of $\widetilde G_S$. More generally, a closed subscheme $Z\subset S$ is distinguished if
\[
\varphi(Z^{\sigma})=Z
\]
for every such pair $(\sigma,\varphi)$; see \cite[Section~7]{bresciani-structure}.

If $p$ is distinguished, the map $\Spec K\to S$ defined by $p$ is $\widetilde G_S$-equivariant. Passing to the quotients in~\eqref{eq:semilinear-quotients} gives a section
\begin{equation}\label{eq:distinguished-section}
s_p:\cG_S\longrightarrow\cS
\end{equation}
of the universal surface.

\section{Del Pezzo surfaces of degree four}
We now consider degree $4$. The anticanonical bundle of a smooth del Pezzo surface $S$ of degree $4$ is very ample and realizes $S$ as a complete intersection of two quadrics
\[
S=Q_0\cap Q_1\subset\bP^4_K.
\]
The pencil $\lambda Q_0+\mu Q_1$ contains exactly five singular quadrics. They form a reduced effective divisor
\[
D_S\subset\bP^1_K
\]
of degree $5$. This construction is independent of the choice of generators $Q_0,Q_1$, up to the natural action of $\PGL_2(K)$.

\begin{lemm}\label{lem:degree-four-discriminant}
The isomorphism class of $S$ is determined by the projective isomorphism class of $D_S$. Moreover,
\[
k_S=k_{D_S},
\]
where $k_{D_S}$ is the field of moduli of $D_S$ as a reduced divisor on $\bP^1$.
\end{lemm}

\begin{proof}
Over $K$, two quadratic forms whose common zero locus is smooth may be simultaneously diagonalized. Thus $S$ admits equations
\begin{equation}\label{eq:degree-four-diagonal}
\sum_{i=0}^4x_i^2=0,
\qquad
\sum_{i=0}^4a_ix_i^2=0,
\end{equation}
where the $a_i$ are pairwise distinct; see, for instance,~\cite[Chapter~8]{dolgachev-classical}. The five points $[-a_i:1]$ are precisely the singular members of the pencil. Conversely, a projective transformation carrying the unordered set of the $a_i$ to another such set, followed by rescaling the coordinates, gives an isomorphism between the corresponding intersections of quadrics. Hence the projective isomorphism class of $D_S$ determines $S$.

Every semilinear isomorphism $S^{\sigma}\simeq S$ induces a projective isomorphism $D_S^{\sigma}\simeq D_S$, and the converse follows from the preceding normal form. The two stabilizer subgroups of $\Gal(K/k)$ are therefore equal, proving $k_S=k_{D_S}$.
\end{proof}

\begin{prop}\label{prop:degree-four}
Every smooth del Pezzo surface of degree $4$ over $K$ is defined over its field of moduli.
\end{prop}

\begin{proof}
Set $k_0=k_S=k_{D_S}$ by Lemma~\ref{lem:degree-four-discriminant}. By~\cite[Theorem~1.1]{Marinatto2013}, a finite subset of odd cardinality of $\bP^1_K$ is defined over its field of moduli. We may therefore choose a reduced divisor
\[
D_0\subset\bP^1_{k_0}
\]
of degree $5$ whose base change to $K$ is projectively isomorphic to $D_S$.

Since $k_0$ is infinite, after a change of coordinates we may assume that the point at infinity does not belong to $D_0$. Choose a separable polynomial $f(t)\in k_0[t]$ of degree $5$ whose zero scheme is $D_0$, and put
\[
A=k_0[t]/(f).
\]
Since $f$ is separable, $f'(t)$ is invertible in $A$. Define quadratic forms on the five dimensional $k_0$-vector space $A$ by
\[
q_0(x)=\operatorname{Tr}_{A/k_0}\left(\frac{x^2}{f'(t)}\right),
\qquad
q_1(x)=\operatorname{Tr}_{A/k_0}\left(\frac{tx^2}{f'(t)}\right).
\]
Let
\[
S_0=\{q_0=q_1=0\}\subset\bP(A).
\]
If $\alpha_0,\ldots,\alpha_4$ are the roots of $f$ in $K$, then after extending scalars to $K$ the two forms become
\[
\sum_{i=0}^4\frac{x_i^2}{f'(\alpha_i)},
\qquad
\sum_{i=0}^4\frac{\alpha_ix_i^2}{f'(\alpha_i)}.
\]
After rescaling the coordinates, these are of the form~\eqref{eq:degree-four-diagonal}. In particular, $S_0$ is smooth, its discriminant divisor is $D_0$, and $(S_0)_K\simeq S$ by Lemma~\ref{lem:degree-four-discriminant}.
\end{proof}

\section{Cubic surfaces: exceptional lines}
\subsection{The anticanonical embedding and its normal line}
For a smooth cubic surface $S$, the anticanonical bundle is intrinsically defined and very ample. Set
\[
V_S=\H^0(S,-K_S)^\vee.
\]
Then $\dim V_S=4$, and the anticanonical embedding is
\[
S\hookrightarrow\bP(V_S)\simeq\bP^3_K.
\]
Let $\varphi: S^{\sigma}\xrightarrow{\sim} S$ be a $\sigma$-semilinear isomorphism, where $\sigma \in \Gal(K/k_S)$. Because
\[
\varphi^*K_S\simeq K_{S^{\sigma}},
\]
the morphism $\varphi$ induces a $\sigma$-semilinear isomorphism $\bP^{3,\sigma}_K\to \bP^3_K$, making the following diagram commute. 
\begin{center}
    \begin{tikzcd}
S^{\sigma} \arrow[r] \arrow[d, hook] & S \arrow[d, hook] \\
{\bP^{3,\sigma}_K} \arrow[r]         & \bP^3_K          
\end{tikzcd}
\end{center}
Consequently, $\varphi$ carries any line on $S^{\sigma}$ to a line on $S$.

Over $\cG_S$, using the anticanonical divisor of the universal surface $\cS\ra \cG_S$, the space $V_S$ forms a rank $4$ vector bundle $\cV$ over $\cG_S$, so that the universal surface is embedded as a relative cubic hypersurface $\cS\hookrightarrow\bP(\cV)$.

Let $p\in S(K)$ be distinguished in the sense of~\eqref{eq:intrinsic-point-galois}. The section $s_p$ of~\eqref{eq:distinguished-section} allows us to define the normal line
\[
N_p:=s_p^*N_{\cS/\bP(\cV)},
\]
which is a line bundle on $\cG_S$. Indeed, since $\cS$ is a smooth relative cubic surface in $\bP(\cV)$, restriction of the relative tangent sequence along $s_p$ gives an exact sequence of vector bundles on $\cG_S$,
\[
0\to s_p^*T_{\cS/\cG_S}
\to s_p^*T_{\bP(\cV)/\cG_S}
\to N_p\to0.
\]
The first two terms have ranks $2$ and $3$, so the quotient $N_p$ has rank $1$. On an object $T\to\cG_S$, represented by a family $S_T\subset\bP^3_T$ together with the induced section $p_T$, its fiber is
\[
(N_p)_T=T_{\bP^3_T/T,p_T}/T_{S_T/T,p_T}.
\]
Thus $N_p$ is exactly the one dimensional normal direction to the surface at the distinguished point. 

\begin{lemm}\label{lem:normal-character}
Let $g\in \Aut(S)$ be represented in anticanonical coordinates by a diagonal matrix with eigenvalues $a_0,\ldots,a_3$. Suppose $p=[e_j]$ and the tangent plane $T_{S,p}$ is given by $x_m=0$. Then the character of $g$ on $N_p$ is
\[
\frac{a_m}{a_j}.
\]
\end{lemm}

\begin{proof}
On the affine tangent space to $\bP^3$ at $[e_j]$, the coordinate direction $x_r$ has eigenvalue $a_r/a_j$. The normal line is the quotient of the ambient tangent space by $T_{S,p}$, hence is represented by the $x_m$-direction.
\end{proof}

\subsection{The 27 lines and the Weyl group}
The Hilbert scheme of lines on $S$ is a reduced finite scheme
\[
L_S=F_1(S)
\]
of length $27$, equipped with its incidence relation. The automorphism group of the standard incidence configuration is the Weyl group $W(E_6)$ of order $51840$, and $\Aut(S)$ can be embedded into $W(E_6)$ as a finite subgroup; see~\cite[\S~8.3]{beauville-cubic-surfaces} and \cite{dolgachev-duncan}.

In this paper an \emph{exceptional line} means one of these $27$ lines in the anticanonical model. Equivalently, it is a $(-1)$-curve $E\subset S$, characterized intrinsically by
\[
E^2=-1,\quad (-K_S)\mathbin{\cdot}E=1.
\]

\subsection{Eckardt configurations}
A \emph{projective reflection} of $\bP(V_S)$ is the projective transformation induced by an order-two linear transformation of $V_S$ that fixes a hyperplane pointwise; the remaining eigenline determines its center; an automorphism of $S$ is a \emph{reflection} if it is the restriction of a projective reflection on $S$.

An Eckardt point is a point through which three exceptional lines on $S$ pass. Equivalently, the tangent-plane section (not the tangent space itself) is
\[
T_{S,p}\cap S=\ell_1\cup\ell_2\cup\ell_3,
\]
with all three lines meeting at $p$; see~\cite[Section~1.2]{dolgachev-duncan}. In characteristic zero, Eckardt points are in bijection with reflections of $S$: the Eckardt point $p$ is the center of the corresponding projective reflection~\cite[Section~1.2]{dolgachev-duncan}.

Let $E(S)$ be the finite set of Eckardt points. There are two geometrically different ways in which several Eckardt points may be collinear.
\begin{itemize}
\item An \emph{exceptional line of the Eckardt configuration} is an exceptional line $\ell\subset S$ containing two Eckardt points. In characteristic zero such a line contains at most two Eckardt points. If $p,q\in E(S)$ lie on $\ell$, then the reflections centered at $p$ and $q$ commute; their product is an involution of class $2B$.
\item A \emph{trihedral line} is a line $m\subset\bP(V_S)$ that is not contained in $S$ and contains three Eckardt points. Since $m\not\subset S$, B\'ezout's theorem gives a zero-dimensional intersection $m\cap S$ of length $3$. Thus, once two distinct Eckardt points $p,q$ on $m$ are known, the third intersection point is forced. In the configurations below this third point is again an Eckardt point. The three corresponding reflections generate a copy of $S_3$ and permute the three points. This is the geometry underlying configuration $\mathsf{Eck}_3$.
\end{itemize}
Following~\cite[Section~13.1]{dolgachev-duncan}, the \emph{Eckardt configuration} consists not only of the finite set $E(S)$, but also of the two incidence relations just described: which pairs of Eckardt points lie on an exceptional line of $S$, and which triples lie on a trihedral line not contained in $S$. All parts of this structure are intrinsic in the following sense: the anticanonical embedding, the scheme of exceptional lines, projective collinearity, and containment in $S$ are preserved by every Galois-semilinear isomorphism.

It is useful to keep several examples in mind. In $\mathsf{Eck}_2$, the two Eckardt points are joined by an exceptional line. In $\mathsf{Eck}_3$, all three points lie on one trihedral line. In $\mathsf{Eck}_4$, there is a distinguished point $t$ joined to each of the other three points by exceptional lines, while those other three points lie on a trihedral line. In $\mathsf{Eck}_6$, the six points may be labeled by the six edges of a tetrahedron: the three edges incident with a vertex form a trihedral triple, while opposite edges form an exceptional pair. In $\mathsf{Eck}_9$, the nine Eckardt points and twelve trihedral lines form the Hesse configuration $(9_4,12_3)$: every trihedral line contains three points and every point lies on four such lines.

To avoid confusion with cyclic groups, we write $\mathsf{Eck}_n$ for the Eckardt configuration having $n$ points, different from the notation as in \cite{dolgachev-duncan}. The strata notation is that of~\cite[Table~5]{dolgachev-duncan}, while the full automorphism groups are listed in~\cite[Table~1]{dolgachev-duncan}.

\begin{center}
\small
\setlength{\tabcolsep}{3pt}
\begin{tabularx}{\textwidth}{c c c >{\raggedright\arraybackslash}X >{\raggedright\arraybackslash}X}
\toprule
configuration&$|E(S)|$&reflections&full automorphisms&stratum/strata\\
\midrule
$\mathsf{Eck}_0$&0&$1$&$1$&$1A$\\
$\mathsf{Eck}_1$&1&$C_2$&$C_2,\ C_4,\ C_8$&$2A,4A,8A$\\
$\mathsf{Eck}_2$&2&$C_2^2$&$C_2^2$&$2B$\\
$\mathsf{Eck}_3$&3&$S_3$&$S_3$&$3D$\\
$\mathsf{Eck}_4$&4&$S_3\times C_2$&$S_3\times C_2$&$6E$\\
$\mathsf{Eck}_6$&6&$S_4$&$S_4$&$4B$\\
$\mathsf{Eck}_9$&9&$H_3(3)\rtimes C_2$&$H_3(3)\rtimes C_2,\ H_3(3)\rtimes C_4$&$3A$ locus; special point $12A$\\
$\mathsf{Eck}_{10}$&10&$S_5$&$S_5$&$5A$\\
$\mathsf{Eck}_{18}$&18&$C_3^3\rtimes S_4$&$C_3^3\rtimes S_4$&$3C$\\
\bottomrule
\end{tabularx}
\end{center}
where $H_3(3)$ is the Heisenberg group of order $27$.

\section{Cubic surfaces: descent from Eckardt configurations}
In this section, we prove Theorem~A in degree $3$.
\subsection{No Eckardt points}
\begin{prop}\label{prop:1A}
If $S$ has automorphism type $1A$, then $S$ descends to $k_S$.
\end{prop}

\begin{proof}
The inertia of $\cG_S$ is trivial, so $\cG_S\simeq\Spec k_S$.
\end{proof}

\subsection{A unique Eckardt point}
Assume that $S$ has configuration $\mathsf{Eck}_1$, with unique Eckardt point $p$. The Eckardt condition is preserved by base change and by every isomorphism $S^{\sigma}\xrightarrow{\sim}S$. Uniqueness therefore implies that $p^{\sigma}$ is carried to $p$, so $p$ is intrinsically distinguished in the sense of~\eqref{eq:intrinsic-point-galois}. The normal bundle $N_p$ over $\cG_S$ is consequently defined.

\begin{prop}\label{prop:C1}
If $S$ has automorphism type $2A$, $4A$, or $8A$, then $S$ descends to $k_S$.
\end{prop}

\begin{proof}
We compute the character on $N_p$ using the normal forms of~\cite[Lemmas~11.4 and~12.12]{dolgachev-duncan}.

For type $2A$, a generator is represented by
\[
\diag(1,1,1,-1),\quad p=[0:0:0:1].
\]
All three ambient tangent eigenvalues at $p$ are $-1$, so the normal character is $-1$.

For type $4A$, one may use the normal form equation
\[
x_3^2x_2+x_2^2x_0+x_1(x_1-x_0)(x_1-cx_0)=0
\]
with generator
\[
g=\diag(1,1,-1,i),\qquad p=[0:0:0:1].
\]
The tangent plane at $p$ is $x_2=0$. By Lemma~\ref{lem:normal-character}, the character on $N_p$ is
\[
\frac{-1}{i}=i,
\]
which has order $4$.

For type $8A$, the surface has the normal form
\[
x_0^3+x_0x_3^2-x_1x_2^2+x_1^2x_3=0,
\]
which is defined over $\bQ\subset k$. We can still study its reflections. A generator is
\[
g=\diag(1,\epsilon^6,\epsilon,\epsilon^4),
\]
where $\epsilon$ is a primitive eighth root of unity. The unique Eckardt point is
\[
p=[0:0:1:0],
\]
and the tangent plane is $x_1=0$. Hence the normal character is
\[
\frac{\epsilon^6}{\epsilon}=\epsilon^5,
\]
which has order $8$.

Thus $N_p$ is a faithful line bundle in all three cases. Every faithful one-dimensional representation is neutral by~\cite[Proposition~5.1]{neutral-representation-dim-leq3}; hence $\cG_S$ is neutral and $S$ descends to $k_S$.
\end{proof}

\subsection{Two Eckardt points}
Assume that the Eckardt configuration is $\mathsf{Eck}_2$, and let $P=\{p_1,p_2\}\subset S$ be the reduced geometric Eckardt locus. For every $\sigma\in G_k$ and every isomorphism $\varphi:S^{\sigma}\xrightarrow{\sim}S$, the map $\varphi$ preserves the Eckardt condition and therefore
\begin{equation}\label{eq:eck2-semilinear-invariance}
\varphi(P^{\sigma})=P.
\end{equation}
Thus the unordered pair, although its individual points need not be Galois fixed, is invariant under the entire semilinear group $\widetilde G_S$. It consequently defines a finite \'etale morphism of degree $2$
\[
q:\cP=[P/\widetilde G_S]\longrightarrow
\cG_S=[\Spec K/\widetilde G_S].
\]

\begin{prop}\label{prop:2B}
If $S$ has automorphism type $2B$, so that $\Aut(S)\simeq C_2^2$, then $S$ descends to $k_S$.
\end{prop}

\begin{proof}
Let
\[
\cL:=\cO_{\bP(\cV)}(-1)|_{\cP}
\qquad\text{and}\qquad
\cW:=q_*\cL.
\]
The relative Eckardt locus $\cP$ is a closed substack of the universal surface $\cS\subset\bP(\cV)$. Pulling back the tautological line bundle $\cO_{\bP(\cV)}(-1)$ along $\cP\to\bP(\cV)$ therefore gives the line bundle $\cL$ on $\cP$. We now explain why its pushforward is a vector bundle on $\cG_S$.

For an object $T\to\cG_S$, set $\cP_T=\cP\times_{\cG_S}T$. The morphism $q_T:\cP_T\to T$ is finite \'etale of degree $2$. After an \'etale cover $T'\to T$, it splits as a disjoint union of two copies of $T'$. Over this cover,
\[
(q_*\cL)_{T'}\simeq
\cL|_{T'_1}\oplus\cL|_{T'_2},
\]
which is locally free of rank $2$. Local freeness and rank descend in the fpqc topology, and finite pushforward is compatible with base change here. Hence $\cW=q_*\cL$ is a rank $2$ vector bundle on $\cG_S$. On the geometric object $S$, after choosing an ordering of the two points, its fiber is
\[
\cW_K\simeq
\cO_{\bP(V_S)}(-1)|_{p_1}
\oplus
\cO_{\bP(V_S)}(-1)|_{p_2}.
\]
However $\cW$ itself can not be written as a direct sum of two rank one bundles, since a semilinear automorphism may exchange the two summands.

By~\cite[Lemma~9.5, Proposition~9.8, and Remark~9.9]{dolgachev-duncan}, every type $2B$ surface in characteristic zero admits anticanonical coordinates in which
\[
S:\quad x_0^2(x_2+c_0x_3)+x_1^2(c_1x_2+x_3)+x_2x_3(x_2+x_3)=0,
\]
the two Eckardt points are
\[
p_1=[1:0:0:0],\qquad p_2=[0:1:0:0],
\]
and the two generating reflections act by independent sign changes of $x_0$ and $x_1$. Consequently, the vector bundle $\cW\ra \cG_S$ is a twisted representation of
\[
C_2^2=\left\{\begin{pmatrix}\pm1&0\\0&\pm1\end{pmatrix}\right\}\subset\GL_2(K),
\]
up to conjugation; see \cite{bresciani2026neutralrepresentationsfinitediagonalizable}. It is faithful since $\Aut_K(S)\cong C_2^2$, and it is neutral by~\cite[Theorem~5.3]{neutral-representation-dim-leq3}.

The rank $2$ bundle $\cW$ therefore gives a faithful morphism from $\cG_S$ to $B\GL_2$ with neutral inertia embedding, so $\cG_S$ is neutral and $S$ descends to $k_S$.
\end{proof}

\subsection{Three, six, and ten Eckardt points}
For configuration $\mathsf{Eck}_3$, the three Eckardt points lie on one trihedral line, and the full automorphism group is $S_3$. The $\mathsf{Eck}_6$ configuration is identified with the six edges of a tetrahedron, and its full automorphism group is $S_4$. The $\mathsf{Eck}_{10}$ configuration is the Clebsch configuration, which recovers the five faces of the Sylvester pentahedron; its full automorphism group is $S_5$.

\begin{prop}\label{prop:3D}\label{prop:4B5A}
If $S$ has automorphism type $3D$, $4B$, or $5A$, then $S$ descends to $k_S$.
\end{prop}

\begin{proof}
In these three cases, respectively,
\[
\Aut_K(S)\simeq S_n,
\quad n=3,4,5.
\]
For $n=3,4,5$, the symmetric group $S_n$ has trivial center and every automorphism of $S_n$ is inner. Hence
\[
\Aut(S_n)\to\operatorname{Out}(S_n)
\]
is split. By \cite[Proposition 4.2]{bresciani-vistoli}, $\cG_S$ is neutral, and $S$ descends to $k_S$.
\end{proof}

\subsection{Four Eckardt points}
The configuration $\mathsf{Eck}_4$ consists of a point $t$ and three further points $p_1,p_2,p_3$. The point $t$ is the common point of three exceptional lines on $S$, each joining it to one of the $p_i$, whereas $p_1,p_2,p_3$ lie on a trihedral line not contained in $S$. This characterizes $t$ uniquely among the four Eckardt points. Since the characterization is preserved by every isomorphism $\varphi:S^{\sigma}\xrightarrow{\sim}S$, one has
\[
\varphi(t^{\sigma})=t.
\]
Thus $t$ is distinguished, in the sense of~\eqref{eq:intrinsic-point-galois}, and the normal line $N_t$ is defined on $\cG_S$. The full geometric automorphism group is $S_3\times C_2$, whose center is the $C_2$ generated by the reflection centered at $t$.

\begin{prop}\label{prop:6E}
If $S$ has automorphism type $6E$, then $S$ descends to $k_S$.
\end{prop}

\begin{proof}
For every object $\xi$ of $\cG_S$, the center of $\Aut(\xi)$ is preserved by pullback and by conjugation. Rigidify $\cG_S$ along $Z=C_2\subset S_3\times C_2$ in the sense of~\cite[Appendix~C]{abramovich-graber-vistoli}:
\[
r:\cG_S\to
\cH:=\cG_S\!\sslash Z.
\]
The morphism $r$ is a $C_2$-gerbe and the geometric inertia of $\cH$ is
\[
(S_3\times C_2)/C_2\simeq S_3.
\]
Again, using the fact that the group $S_3$ has trivial center and $\Aut(S_3)\to\operatorname{Out}(S_3)$ is split, by \cite[Proposition 4.2]{bresciani-vistoli}, we have $\cH$ is neutral. Choose an object $h:\Spec k_S\to\cH$ and form the fiber
\[
\cK:=\cG_S\times_{\cH,h}\Spec k_S.
\]
Then $\cK$ is a gerbe with geometric inertia $C_2$.

The normal line $N_t$ on $\cG_S$ restricts to a line bundle on $\cK$. Its geometric inertia representation is faithful: the nontrivial element of $C_2$ is the projective reflection centered at $t$, and it acts on the normal direction by $-1$. Therefore $\cK$ is neutral by~\cite[Proposition~5.1]{neutral-representation-dim-leq3}, and hence $\cG_S$ is neutral.
\end{proof}

\subsection{The Fermat surface}
\begin{prop}\label{prop:Fermat}
If $S$ has automorphism type $3C$, then $S$ descends to $k_S$.
\end{prop}

\begin{proof}
By~\cite[Lemma~10.14 and \S 1.3]{dolgachev-duncan}, $S$ is the Fermat cubic surface. It has the model
\[
x_0^3+x_1^3+x_2^3+x_3^3=0
\]
over $\bQ\subset k_S$.
\end{proof}

\subsection{Cyclic cubic surfaces}
Now comes the last case. A smooth cubic surface is cyclic if it admits a cyclic triple-cover presentation
\[
S:\quad w^3=F(x_0,x_1,x_2),
\]
where $F=0$ is a smooth plane cubic. The deck transformation acts by $w\mapsto\zeta_3w$.

Dolgachev and Duncan prove that a smooth cubic surface is cyclic exactly when it admits an automorphism of class $3A$~\cite[Lemma~10.5]{dolgachev-duncan}. Here ``admits'' is an existence condition; it does not mean that the full automorphism group is the generic group on the $3A$ locus. In particular, the unique characteristic-zero surface of type $12A$ is also cyclic: it is the special member of the cyclic $3A$ family for which the branch plane cubic has an additional automorphism of order $4$. Consequently its automorphism group contains the class-$3A$ deck subgroup as well as elements of class $12A$; see~\cite[Section~1.4 and Lemma~12.15]{dolgachev-duncan}. 

\begin{lemm}\label{lem:intrinsic-cyclic}
Let $S$ be a non-Fermat cyclic cubic surface in characteristic zero. Then its deck subgroup $C_3\subset\Aut(S)$ is uniquely determined by $S$. Consequently, every isomorphism between two such surfaces carries the deck subgroup of one to the deck subgroup of the other.
\end{lemm}

\begin{proof}
On the $3A$ stratum, the reflections associated with the nine Eckardt points generate $H_3(3)\rtimes C_2$, and the deck subgroup is its center by~\cite[Lemma~10.8]{dolgachev-duncan}. This is the full automorphism group on the general $3A$ stratum. On the exceptional $12A$ surface, the uniqueness of the cyclic surface structure is part of the classification argument in~\cite[Lemma~12.15]{dolgachev-duncan}. The only cyclic surface where uniqueness fails is the Fermat surface, which is excluded.
\end{proof}

\begin{lemm}\label{lem:fields-cyclic}
Let $S$ be a non-Fermat cyclic cubic surface and let $E\subset\bP^2_K$ be the branch plane cubic of its cyclic structure. Then
\[
k_S=k_E,
\]
where $k_E$ is the field of moduli of $E$ as a plane cubic.
\end{lemm}

\begin{proof}
For $\sigma\in \Gal(K/k_S)$, every isomorphism $S^{\sigma}\simeq S$ preserves the unique deck subgroup by Lemma~\ref{lem:intrinsic-cyclic}; hence it descends to a projective isomorphism $E^{\sigma}\simeq E$. Thus $k_E\subseteq k_S$.

Conversely, suppose that $E^{\sigma}$ and $E$ are projectively isomorphic, i.e., isomorphic via an isomorphism of $\bP_K^2$. Choose equations $F^{\sigma}$ and $F$ and a projective transformation $A$ such that
\[
F^{\sigma}(Ax)=cF(x)
\]
for some $c\in K^\times$. Choose $\lambda\in K^\times$ with $\lambda^3=c$. Then
\[
(x,w)\mapsto(Ax,\lambda w)
\]
lifts the projective isomorphism of branch curves to an isomorphism of cyclic cubic surfaces. Hence $k_S\subseteq k_E$.
\end{proof}

\begin{lemm}\label{lem:plane-cubics}
Every smooth plane cubic over $K$ has a plane model over its field of moduli as a plane cubic.
\end{lemm}

\begin{proof}
The projective isomorphism class of a smooth plane cubic over an algebraically closed field is determined by its $j$-invariant (any two degree-$3$ polarizations on a genus-one curve differ by a translation). Hence its field of moduli is $F=k(j)$. A standard Weierstrass equation over $F$ with invariant $j$ gives the required plane model.
\end{proof}

\begin{prop}\label{thm:cyclic}
Every smooth cyclic cubic surface over $K$ is defined over its field of moduli.
\end{prop}

\begin{proof}
The Fermat case is Proposition~\ref{prop:Fermat}. Assume that $S$ is not Fermat and put $k_0=k_S=k_E$ by Lemma~\ref{lem:fields-cyclic}. By Lemma~\ref{lem:plane-cubics}, choose a homogeneous cubic
\[
F_0\in k_0[x_0,x_1,x_2]
\]
whose zero locus becomes projectively isomorphic to $E$ over $K$. Then
\[
S_0:\quad w^3=F_0(x_0,x_1,x_2)
\]
is a smooth cubic surface over $k_0$, and $S_0\otimes_{k_0}K\simeq S$.
\end{proof}

\begin{proof}[Proof of Theorem~A in degree $3$]
By the classification in~\cite{dolgachev-duncan}, every smooth cubic surface has one of the automorphism types appearing in the table of Section~5. Type $1A$ is treated in Proposition~\ref{prop:1A}; types $2A$, $4A$, and $8A$ in Proposition~\ref{prop:C1}; type $2B$ in Proposition~\ref{prop:2B}; types $3D$, $4B$, and $5A$ in Proposition~\ref{prop:3D}; type $6E$ in Proposition~\ref{prop:6E}; and type $3C$ in Proposition~\ref{prop:Fermat}. The remaining $3A$ locus, including the special type $12A$, consists of cyclic cubic surfaces and is treated in Proposition~\ref{thm:cyclic}.
\end{proof}

\section{Del Pezzo surfaces of degrees one and two}
\subsection{Degree two}
Let $S$ be a smooth del Pezzo surface of degree $2$ over $K$. The anticanonical linear system defines a finite morphism of degree $2$
\[
\pi:S\to\bP^2_K
\]
branched along a smooth plane quartic $C\subset\bP^2_K$. Equivalently, $S$ is a quartic hypersurface
\[
S:\quad u^2=F_4(x_0,x_1,x_2)
\subset\bP_K(1,1,1,2).
\]
The deck transformation $u\mapsto-u$ is the Geiser involution; it is intrinsically determined by $S$.

\begin{lemm}\label{lem:degree-two-branch}
With notation as above, $k_S=k_C$. Moreover, $S$ is defined over $k_S$ if and only if $C$ is defined over $k_C$.
\end{lemm}

\begin{proof}
Every isomorphism $S^{\sigma}\simeq S$ identifies the anticanonical maps, and hence induces a projective isomorphism $C^{\sigma}\simeq C$. Conversely, suppose that $C^{\sigma}$ and $C$ are projectively isomorphic. If $F_4^{\sigma}(Ax)=cF_4(x)$ for some $A\in\PGL_3(K)$ and $c\in K^\times$, then choosing $\lambda\in K^\times$ with $\lambda^2=c$ gives an isomorphism
\[
(x,u)\mapsto(Ax,\lambda u)
\]
between the double covers. Thus $k_S=k_C$.

If $C$ has a model over $k_C$, its canonical linear system gives a plane quartic model over $k_C$, and the corresponding double cover gives a model of $S$. Conversely, a model of $S$ over $k_S$ has an anticanonical morphism to
\[
\bP\bigl(\H^0(S,-K_S)^\vee\bigr)\simeq\bP^2_{k_S},
\]
whose branch divisor is a model of $C$.
\end{proof}

We recall the following explicit example of Artebani and Quispe~\cite[Section~4]{artebani-quispe}. Let $C\subset\bP^2_\bC$ be the plane quartic
\begin{equation}\label{eq:quartic-counterexample}
\begin{split}
y^4&+y^2(x-a_1z)(x+a_1^{-1}z)\\
&+(x-a_2z)(x+\overline{a}_2^{-1}z)
(x-a_3z)(x+\overline{a}_3^{-1}z)=0,
\end{split}
\end{equation}
where
\[
a_1=1,\quad a_2=1-i,\quad a_3=2(i-1).
\]
The curve $C$ is smooth and
\[
\Aut(C)=\langle\nu\rangle\cong C_2,
\quad
\nu(x:y:z)=(x:-y:z).
\]
There is an isomorphism
\[
\mu:C\to\overline C,
\quad
\mu(x:y:z)=(-z:iy:x),
\]
and every isomorphism from $C$ to $\overline C$ is equal to $\mu$ or $\mu\nu$. In both cases the composition with the conjugate is $\nu$, rather than the identity. Thus the field of moduli of $C$ relative to $\bC/\bR$ is $\bR$, but Weil's cocycle condition \cite[Theorem 1]{Weil1956} shows that $C$ is not defined over $\bR$.

\begin{rema}
    By \cite[Corollary 5.7]{bresciani-plane}, we know that the curve $C$ fails to be defined over $k_C$ only if $\Aut(C)\cong C_2$. Thus the Artebani–Quispe example has the only possible automorphism group for such a counterexample.
\end{rema}

\begin{prop}\label{prop:degree-two-counterexample}
There exists a smooth del Pezzo surface of degree $2$ over $\bC$ whose field of moduli relative to $\bC/\bR$ is $\bR$, but which is not defined over $\bR$.
\end{prop}

\begin{proof}
Let $F_4$ be the quartic form in~\eqref{eq:quartic-counterexample}, and let
\[
S:\quad u^2=F_4(x,y,z)
\subset\bP_\bC(1,1,1,2).
\]
This is a smooth del Pezzo surface of degree $2$. By Lemma~\ref{lem:degree-two-branch}, its field of moduli is $\bR$. If $S$ had a real model, the branch curve of its anticanonical morphism would give a real model of $C$, a contradiction.
\end{proof}

\subsection{Degree one}
A smooth del Pezzo surface of degree $1$ over $\bC$ has a weighted anticanonical model
\begin{equation}\label{eq:degree-one-weighted}
S:\quad
w^2+z^3+F_4(x,y)z+F_6(x,y)=0
\subset\bP_\bC(1,1,2,3),
\end{equation}
where $F_4$ and $F_6$ are binary forms of degrees $4$ and $6$; see~\cite[Section~6.7]{dolgachev-iskovskikh}. The deck transformation
\[
\beta(x,y,z,w)=(x,y,z,-w)
\]
of the morphism defined by $|-2K_S|$ is the Bertini involution.

For $m=4,6$, let $W_m$ be the real vector space of binary forms satisfying
\begin{equation}\label{eq:quaternionic-binary-forms}
F_m(-y,x)=\overline{F_m(x,y)}.
\end{equation}
If
\[
F_m(x,y)=\sum_{r=0}^mc_rx^{m-r}y^r,
\]
then~\eqref{eq:quaternionic-binary-forms} is equivalent to
\[
c_{m-r}=(-1)^r\overline{c_r}.
\]
In particular, $W_m$ is a real form of the complex vector space of binary forms of degree $m$, and is Zariski dense in it.

\begin{prop}\label{prop:degree-one-counterexample}
There exists a smooth del Pezzo surface of degree $1$ over $\bC$ whose field of moduli relative to $\bC/\bR$ is $\bR$, but which is not defined over $\bR$.
\end{prop}

\begin{proof}
Choose a general pair
\[
(F_4,F_6)\in W_4\oplus W_6.
\]
We may assume that the surface~\eqref{eq:degree-one-weighted} is smooth and that
\[
\Aut(S)=\langle\beta\rangle\simeq C_2.
\]
Indeed, smoothness and the condition that the full automorphism group is generated by the Bertini involution hold on a nonempty Zariski open subset of the space of pairs $(F_4,F_6)$; the latter is the general type in the classification~\cite[Table~8, type~XXI]{dolgachev-iskovskikh}. Since $W_4\oplus W_6$ is Zariski dense, it meets this open subset.

Condition~\eqref{eq:quaternionic-binary-forms} implies that
\[
\varphi:\overline S\longrightarrow S,
\quad
(x,y,z,w)\mapsto(-y,x,z,w)
\]
is an isomorphism. The coordinate transformation defining $\varphi$ is real, and hence
\[
\varphi\overline\varphi(x,y,z,w)=(-x,-y,z,w).
\]
In $\bP(1,1,2,3)$, weighted multiplication by $-1$ identifies the latter transformation with
\[
(x,y,z,w)\mapsto(x,y,z,-w)=\beta(x,y,z,w).
\]
Therefore
\[
\varphi\overline\varphi=\beta\neq1.
\]
It follows that $S\simeq\overline S$, so its field of moduli relative to $\bC/\bR$ is $\bR$. Every isomorphism $\overline S\to S$ is equal to either $\varphi$ or $\beta\varphi$. Since $\beta$ is central and defined over $\bR$, the composition of either isomorphism with its conjugate is again $\beta$. Thus no isomorphism satisfies Weil's cocycle condition, and $S$ is not defined over $\bR$.
\end{proof}

\printbibliography
\end{document}